\documentclass{article}
\usepackage{amsfonts}
\usepackage{amsmath}
\usepackage[all,cmtip]{xy}
\usepackage{amssymb}
\usepackage{tikz-cd}
\usepackage{amscd}
\usepackage{cancel}
\usepackage{tabularx}
\usepackage{mathtools}
\usepackage{here}
\usepackage{lipsum}
\usepackage{lineno}
\usepackage{tikz}
\newenvironment{proof}{\noindent{\bf Proof.}}
{\noindent \ \hfill$\Box$\par}
\usepackage[all]{xy}
\usepackage{xcolor}
\newcommand{\bzd}{\hfill\boxed{}}
\usepackage{bbm}
\usepackage{colortbl}
\usepackage{selinput}
\usepackage{hyperref}
\hypersetup{
    colorlinks=true,
    linkcolor=blue,
    filecolor=magenta,
    urlcolor=cyan,
    pdftitle={Overleaf Example},
    pdfpagemode=FullScreen,
    }
\usepackage[a4paper, left=1in, right=1in, top=1in, bottom=1in]{geometry}
\newtheorem{theorem}{Theorem}[section]

\newtheorem{pro}[theorem]{Proposition}
\newtheorem{lem}[theorem]{Lemma}
\newtheorem{cor}[theorem]{Corollary}
\newtheorem{rem}[theorem]{Remark}

\newcommand{\hc}{\circ}
\newcommand{\n}{\{ }
\newcommand{\nn}{\} }
\newcommand{\ca}{\eta}
\newcommand{\cg}{\sigma}
\newcommand{\lt}{\varepsilon}
\newcommand{\ch}{ \bar{\nu}  }

\newcommand{\cn}{\kappa}

\newcommand{\ccs}{\bar{\sigma}}
\newcommand{\ct}{\bar{\kappa}}
\newcommand{\kj}{\phi}
\newcommand{\cx}{\delta}
\newcommand{\w}{\omega}
\newcommand{\wq}{\infty}
\newcommand{\af}{\alpha}

\newcommand{\ord}{\mathrm{ord}}

\newcommand{\m}{\;\mathrm{mod}\,}

\newcommand{\e}{\iota}

\newcommand{\dyd}{\supseteq}

\newcommand{\ty}{ \equiv}

\newcommand{\tg}{\approx}

\newcommand{\jia}{\oplus}
\newcommand{\z}{\mathbb{Z}}

\newcommand{\s}{\Sigma}

\usepackage{array}

\newcolumntype{Y}{>{\centering\arraybackslash}X}

\newcommand{\q}{ \Delta}

\begin{document}
\date{}
\title{\textsc{On the extension problems for  three  33-stem homotopy groups}}
\author{Juxin Yang\thanks{\text{Beijing Institute of Mathematical Sciences and Applications, Beijing, 101408, P.R. China.  yangjuxin@bimsa.com}}  \,\;and Jie Wu\,\thanks{Beijing Institute of Mathematical Sciences and Applications,
Beijing, 101408, P.R. China.
wujie@bimsa.com}}
\maketitle
\begin{abstract}
This paper tackles the extension problems for three far-unsatble homotopy groups $\pi_{39}(S^{6})$, $\pi_{40}(S^{7})$, and $\pi_{41}(S^{8})$ localized at 2, the puzzles    having remained unsolved for forty-five years. By a Toda bracket indexed by 1  included in $\pi_{39}(S^{6}_{(2)})$, which makes better use of the deuspension property of homotopy classes, we address the problems. As a corollary, through  Thomeier's 8-step backward theorem of the metastable homotopy theory, together with the results of Oda, Mukai and Miyauchi, we show a table of the 33-stem homotopy groups $\pi_{33+n}(S^{n}_{(2)})$, ($2\leq n\leq 9$, $n\geq27$). 
\end{abstract}

\section{Introduction}

The  extension problems from the 1979 literature by N. Oda (\cite[pp.\,145]{Oda}), which have gone uncracked for the past 45 years, are as follows:\vspace{-0.4\baselineskip} \[ 
\pi_{39}(S^{6})\tg(\z/2)^{8} \text{\;or\;}(\z/2)^{6}\jia\z/4,\]\[\pi_{40}(S^{7})\tg(\z/2)^{6}  \text{\;or\;}(\z/2)^{4}\jia\z/4\;\] \[\text{\;and\;}\; \pi_{41}(S^{8})\tg\z/8\jia(\z/2)^{8}  \text{\;or\;} \z/8\jia\z/4\jia(\z/2)^{6}\;\; \text{\;localized at 2.}\]
 The key aspect of addressing these problems lies in determining the order of the element $\cn'_{6}\in\pi_{39}(S^{6}_{(2)})$ with nontrivial \textit{Hopf} invariant  $\ch_{11}\ct_{19}$. In this article, utilizing a Toda bracket indexed by 1  included in $\pi_{39}(S^{6}_{(2)})$, namely, $\n\s\e_{5}, \s(\nu_{5}\ca_{8}), \s(\ch_{9}\ct_{17}) \nn_{1}$, we solve  the extension problems. During our discussion with Professor Juno Mukai, he highlights that a Toda bracket using   \textit{id} to define is often regarded as a highly unsuitable tool, as it impedes the tracking of elements via the iterated suspension functor $\s^{k}$. However, we utilized it which is logically sound. Furthermore, it is instrumental in overcoming Oda's extension problems. In some sense, it is  because in practice  the Toda bracket of the form $\n\; id, \s^{n}-,\s^{n}-\nn_{n}\,(n\geq1)$  has been frequently disregarded  that these problems have remained unresolved for such an extended period. It is noteworthy that the  Toda bracket indexed by $n$ with $n\geq1$ makes better use of the desuspension properties of homotopy classes and plays a critical role in the theory of computing unstable homotopy groups.

Homotopy groups occupy a central and foundational position in  homotopy theory, encapsulating the essence of building  spaces. These groups serve as pivotal invariants, providing profound insights into the inherent geometric and algebraic characteristics of spaces. Among the myriad of homotopy groups, those pertaining to spheres hold a particularly preeminent and influential status.  Through meticulous analysis of these groups, researchers have devised sophisticated techniques and unearthed unexpected results, significantly impacting areas such as algebraic topology, differential geometry, and theoretical physics.

Given $k\in\z_{+}$, the groups $ \pi_{n+k}(S^{n})\,(n\geq2)$ are called the \textit{k}-stem homotopy groups of spheres.
The initial systematic and effective computation of homotopy groups of spheres is presented in \cite{Toda} (the \textit{1} to \textit{19}-stems) by H. Toda in 1962, followed subsequently by \cite{20STEM} (the \textit{20}-stem), \cite{2122STEM} (the \textit{21,22}-stems localized at 2)  and \cite{2324STEM} (the \textit{23,24}-stems localized at 2).

In 1979, N. Oda (\cite{Oda}) studies the homomotopy groups of spheres $\pi_{n+k}(S^{k})\,\;(25\leq n\leq 31)$ for all $k\geq2$ and 
$\pi_{n+k}(S^{k})\,(n=32, 33)$ for  $2\leq k\leq 8$ localized at 2.  This work epitomizes an exceptional blend of rigorous logic and artistic flair. Oda employs an array of sophisticated techniques to compute these homotopy groups, manifesting an    extraordinary level of mathematical prowess. The employed methodologies encompass the classical Toda bracket method, the \textit{4}-fold Toda bracket method, the Adams spectral sequence and Im$(J)$-theory. However,  the \textit{33}-stem homotopy groups  $\pi_{39}(S^{6})$, $\pi_{40}(S^{7})$ and $\pi_{41}(S^{8})$ localized at 2 are incompletely determined. 

In 2017,  the determinations of the 
\textit{32}-stem homotopy groups localized at 2, namely, the groups $\pi_{32+k}(S^{k}_{(2)})$ for all $k\geq 2$, are comprehensively concluded by T. Miyauchi and J. Mukai (\cite{32STEM}). The authors provide  a new tool for  determinations of unstable homotopy groups, that is, the left matrix  Toda bracket indexed  by $n$, which makes better use of the desuspension property of homotopy classes if $n\geq1$ and is a key ingredient to compute $\pi_{46}(S^{14}_{(2)})$.   Notably, \cite{32STEM} gives a comprehensive computation of the homotopy group $\pi_{42}(S^{9}_{(2)})$,  a \textit{33}-stem  homotopy group localized at 2. While the three groups $\pi_{33+n}(S^{n}_{(2)})\,(n=6,7,8)$  are still unresolved within \cite{32STEM}.

Herein lies our  main theorem, which  addresses the extension problems posed by Oda (\cite[pp.\,145]{Oda}), (see Proposition \ref{zdl2}).

\begin{theorem}\label{zdly} Localized at 2,  \\\\ \centerline{
    $\pi_{39}(S^{6})\tg\z/4\jia(\z/2)^{6}$, \;$\ord(\cn'_{6})=4;$ 
      $\pi_{40}(S^{7})\tg\z/4\jia(\z/2)^{4};$
      $\pi_{41}(S^{8})\tg\z/8\jia\z/4\jia(\z/2)^{6}$.} 
\end{theorem}

\noindent {\bf Acknowledgement}. This work is supported in part by the start-up research fund from Beijing Institute of Mathematical
Sciences and Applications.  We are  indebted  to Prof. Juno Mukai  and Dr. Xiangxuan Yi for   many fruitful conversations on
this project.
\section{Preliminaries}

\subsection{Notations }
 In this paper, all spaces, maps, homotopy classes are pointed.  Basepoints and constant maps are denoted by $*$,   homotopy classes of constant maps are denoted by $0$. To indicate the domain and codomain, the trivial element in $\pi_{n+k}(S^{n})$ is also denoted by $0_{n}^{(k)}$.

 \indent  Let $\alpha\in\pi_{n}(X),\beta\in\pi_{m}(S^{n})$\, where $n\geq2$ and let $k\in\z$; usually and reasonably, $\alpha\beta$ is the the abbreviation of $\alpha\circ\beta$ and $ k\alpha\circ\beta$ is the the abbreviation of $(k\alpha)\circ\beta$;  in this article we only use the symbol $k\alpha\beta$  to denote $k(\alpha\beta)$.\;So it's necessary to  point out that  $k\alpha\beta\neq k\alpha\circ\beta \;\text{in general}.$  
Of course $k\alpha\beta= k\alpha\circ\beta$ always holds  if $\beta$ is a suspension, or the codomain of $\alpha$  is $S^{7}$ or   a  group-like \textit{H}-space (\cite[p.\,118]{GW}), in particular, a topological group. 

The homotopy group $\pi_{m}(S^{n}_{(2)})$ is also denoted by $\pi_{m}^{n}$. We follow Toda's notations in \cite{Toda} of the generators of 
$\pi_{*}^{n}$, whose notations and  naming 
 convention are also adopted by \cite{20STEM},\cite{2122STEM},\cite{2324STEM},\cite{Oda},\cite{32STEM} and so on. Recall from \cite{Toda} that a set containing a single element is  identified with its element. 
We denote Toda's $E$ by $\s$, the  suspension functor, and we  denote Toda's    $\q$ by $P$, the  boundary homomorphism of the  \textit{EHP} sequence. 
  There is an  advantage of Toda's naming convention, that is, we can examine the commutativity of the unstable composition conveniently (also see \cite[Proposition 3.1]{Toda}). Given $a,b\in\z_{+}$,  let $\mathbbm{x}_{k+a}=\s^{k}\mathbbm{x}_{a}$ and  $\mathbbm{y}_{k+b}=\s^{k}\mathbbm{y}_{b}$\, ($k\geq0$) where $\mathbbm{x}_{a}\in\pi_{*}^{a}$ and $\mathbbm{y}_{b}\in\pi_{\bullet}^{b}$; then,
  $$\mathbbm{x}_{n}\hc \mathbbm{y}_{i}=\pm \mathbbm{y}_{n}\hc\mathbbm{x}_{j}  \;\;\text{ for some}\; i, j\;\;\; \text{if} \;\;n\geq a+b.$$
 For instance, for the  elements  in \cite{Toda},  $\cg_{n}\in\pi_{7+n}^{n}\,(n\geq8) \;\text{and} \;\mu_{n}\in\pi_{9+n}^{n}\,(n\geq3),$ we have  $\cg_{8+3}\mu_{i}=\pm \mu_{8+3}\cg_{j}$, ($\ord( \mu_{3})$=2). Successively,\; $\cg_{11}\mu_{18}=\mu_{11}\cg_{20}$. But  $\cg_{10}\mu_{17}\neq\mu_{10}\cg_{19}$, (see \cite[p.\,156]{Toda}).
 Some common  generators are summarized in \cite[p.\,189]{Toda} and \cite[(1.1),\,p.\,66]{Oguchi}.

  Let $(G,+)$ be an abelian group and $A$ be a subgroup,  $g,g'\in G$; following Oda (\cite{Oda}), we write  $g\ty g'\m A$ if $g-g'\in A.$ If $A$ is generated by $\n a_{ \lambda }\nn _{ \lambda \in \Lambda }$, then ``$\m A$\,'' is  denoted by ``$\m\,a_{ \lambda },\, ( \lambda \in \Lambda )$\,''.

Suppose $U$ is an abelian group and $V$ is a subgroup. If we have a coset $S\in U/V$, then $V$ is often called the  \textit{\textbf{indeterminacy}} of $S$, denoted by $\mathrm{Ind}(S)=V$.

\subsection{Some fundamental facts}
Let us recall the \textit{3}-fold Toda bracket.
Given the following sequence of spaces and homotopy classes with  $Z$ is a suspension such that $\af\hc \s^{n}\beta=\beta\gamma=0$,\vspace{-0.3\baselineskip}
\[\xymatrix@!C=0.83cm{ &W&  \s^{n}X\; & \s^{n}Y  &\s^{n}Z,& 
       \ar"1,4";"1,3" _{\;\s^{n}\beta_{}} \ar"1,5";"1,4" _{ 
 \s^{n}\gamma}
\ar"1,3";"1,2" _{\af}
     }\]
in 1962, Toda (\cite{Toda}) defined the \textit{3}-fold  Toda bracket indexed by $n$ denoted by  $\n\af,\s^{n}\beta,\s^{n}\gamma\nn_{n}$.   Such a Toda bracket is a coset included in $[\s^{n+1}Z,W]$, that is,  $\n\af,\s^{n}\beta,\s^{n}\gamma\nn_{n}\in [\s^{n+1}Z,W]/A,$\noindent \;where the subgroup \begin{equation}
     A=\af\hc \s^{n}[\s Z,X]+[\s^{n+1}Y, W]\hc \s^{n+1}\gamma. \label{ybsind}
 \end{equation} 
It is clear that  
$f\hc \n\af,\s^{n}\beta,\s^{n}\gamma\nn_{n}$
is a coset of $f\hc A$,\;and $\n\af,\s^{n}\beta,\s^{n}\gamma\nn_{n}\hc \s g$\;
      is a coset of $A\hc \s g$, ($f,g:$ maps). For a fixed $n$,  $\n \af,\s^{n}\beta,\s^{n}\gamma\nn_{n}$   depends only on $(\af, \beta,\gamma)$ but not $(\af, \s^{n}\beta,\s^{n}\gamma)$.  It is necessary to point out that  even if\; $\s^{n}\beta'=\s^{n}\beta, \s^{n}\gamma'=\s^{n}\gamma \; \text{and}\; \beta'\gamma'=0$,
      $$\n \af,\s^{n}\beta',\s^{n}\gamma'\nn_{n}\neq\n \af,\s^{n}\beta,\s^{n}\gamma\nn_{n}\;\text{in general},$$
  (see \cite[Remark 3.1]{dd}).
   For more basic properties of Toda brackets, see \cite[pp.\,10-12]{Toda}.

\subsection{Some relations on the generators of $\pi_{*}^{n}$}

In the generator set of $\pi_{34}^{5}$, the generator $\phi_{5}\nu_{29}^{2}$ can be replaced by  $\nu_{5}\cg_{8}\ccs_{15}$. As a matter of fact, taking into account $H(\nu_{5}\cg_{8}\ccs_{15})=H(\kj_{5}\nu_{29}^{2})=\nu_{9}^{2}\ccs_{15}$ (\cite[Proposition 4.3 (2), pp.\,91]{Oda}) and $\pi_{*}^{4}=\nu_{4}\hc\pi_{*}^{7}\jia\s\pi_{*-1}^{3}$ (\cite[(5.6), pp.\,42]{Toda}), together with the results provided on \cite[pp.\,104, pp.\,143]{Oda}), we have the following remark. 

\begin{rem}\label{hjzj}
    \begin{itemize}
        \item [\rm(1)] $\pi_{34}^{5}=\mathrm{span}\n \nu_{5}\cg_{8}\ccs_{15}, \nu_{5}\ct_{8}\nu_{28}^{2},\nu_{5}^{3}\ct_{14},\nu_{5}\ca_{5}\mu_{3,9},\ca_{5}\lt_{6}\ct_{14}\nn\tg(\z/2)^{6}.$
        \item [\rm(2)] $\pi_{35}^{6}=\mathrm{span}\n \cx',\bar{\lt}',\nu_{6}\cg_{9}\ccs_{16},\ca_{6}\lt_{7}\ct_{15}\nn\tg(\z/4)^{2}\jia(\z/2)^{2}.$
        \item [\rm(3)]  $\pi_{36}^{7}=\mathrm{span}\n \cx'',\cg'\lt_{14}\cn_{22}, \cg'\w_{14}\nu_{30}^{2},
 \nu_{7}\cg_{10}\ccs_{17},\ca_{7}\lt_{8}\ct_{16}\nn\tg\z/8\jia(\z/2)^{4}.$
        \item [\rm(4)]  $\pi_{36}^{4}=\mathrm{span}\n \nu_{4}\cx'',\nu_{4}\cg'\lt_{14}\cn_{22}, \nu_{4}\cg'\w_{14}\nu_{30}^{2},
 \nu_{4}^{2}\cg_{10}\ccs_{17},\nu_{4}\ca_{7}\lt_{8}\ct_{16}, \s\kj'\hc\cg_{29},\s\nu'\hc\ca_{7}\lt_{8}\ct_{16},\mu_{3,4}\cg_{29}\nn\\\indent\quad\;\,\tg\z/8\jia(\z/2)^{4}\jia\z/4\jia(\z/2)^{2}.$
    \end{itemize}
\end{rem}
The following lemma is intended to lay the groundwork for the subsequent subsection.
\begin{lem}\label{qmgx}\begin{itemize}
   
    \item[\rm(1)] 
 $\ca_{5}\ch_{6}=\nu_{5}^{3}\in\pi_{14}^{5};\;$ $2\e_{14}\hc\w_{14}=4\w_{14}\in\pi_{30}^{14}.$
   \item[\rm(2)]
   $4(\nu_{4}\hc\cx'')=\nu_{4}\hc4\cx''\ty\nu_{4}\hc \nu_{7}^{3}\ct_{16}=\nu_{4}\hc\ca_{7}\ch_{8}\ct_{16}\m \nu_{4}\hc\nu_{7}\cg_{10}\ccs_{17}$.
    \item[\rm(3)] $P(4\cg_{9}\cg_{16}^{*})=4(\nu_{4}\hc\cx'')\ty\nu_{4}\ca_{7}\ch_{8}\ct_{16}\m\nu_{4}^{2}\cg_{10}\ccs_{17}.$ And so \\$P(4\cg_{9}\cg_{16}^{*})=4(\nu_{4}\hc\cx'')=\nu_{4}\ca_{7}\ch_{8}\ct_{16}+x\nu_{4}^{2}\cg_{10}\ccs_{17}$ for some $x\in\z.$

\item[\rm(4)] $P(\cg_{9}\nu_{16}\ccs_{19})=\nu_{4}^{2}\cg_{10}\ccs_{17}\in\pi^{4}_{36}$
\item[\rm(5)] $4\cg_{9}\cg_{16}^{*}+x\cg_{9}\nu_{16}\ccs_{19}\in P^{-1}(  \nu_{4}\ca_{7}\ch_{8}\ct_{16})$.

   \item[\rm(6)]  $P^{-1}(  \nu_{4}\ca_{7}\ch_{8}\ct_{16})=4\cg_{9}\cg_{16}^{*}+x\cg_{9}\nu_{16}\ccs_{19}+\mathrm{span}\n8\cg_{9}\cg_{16}^{*}\nn.$
    \item[\rm(7)] $\nu_{5}\cg_{8}\nu_{15}\ct_{18}+\s\beta\in \n \nu_{5}\ca_{8}, \ch_{9}\ct_{17},2\e_{37}\nn_{1}$ for some $\beta\in \pi^{4}_{37}.$
    \item[\rm(8)]  $\s^{2}\pi_{37}^{4}=\mathrm{span}\n\mu_{6,6},\;\ca_{6}\mu_{3,7}\cg_{32}\nn\tg\z/2.$ And so for the above element $\beta$, $\s^{2}\beta\in\mathrm{span}\n\mu_{6,6},\;\ca_{6}\mu_{3,7}\cg_{32}\nn.$

\end{itemize}
    
\end{lem}

\begin{proof}\begin{itemize}

    \item[\rm(1)] The first one  is just \cite[(7.3), pp.\,64]{Toda}. For the second one, by \cite[Lemma 12.5, pp.\,155]{Toda}, we have $H(\w_{14})=\nu_{27}$. The final line on \cite[pp.\,159]{Toda} states that $P(\nu_{29})=\pm2\w_{14}$.  Recall from \cite[Proposition 2.2, pp.\,18]{Toda} that $H(\s\af\hc\beta)=\s(\af\wedge \af)\hc H(\beta)$. Then the result follows from \cite[Proposition 2.10 (2)]{Oda} and $H(2\e_{14}\hc\w_{14})=4\nu_{27}\neq0.$
    \item[\rm(2)] By (1), we infer $\nu_{7}^{3}=\ca_{7}\ch_{8}.$  We  need only to show $4\cx''\ty\nu_{7}^{3}\ct_{16}\m\nu_{7}\cg_{10}\ccs_{17}$. This follows from the proof of \cite[(7.5), pp.\,130]{Oda}.
    \item[\rm(3)]  By the first line on \cite[pp.\,144]{Oda} and Remark \ref{hjzj}\,(3), we derive $P(\cg_{9}\cg_{16}^{*})=t\nu_{4}\cx'' +\mathbbm{b}$
where $t$ is odd and $\mathbbm{b}\in\mathrm{span}\n \nu_{4}\cg'\lt_{14}\cn_{22}, \nu_{4}\cg'\w_{14}\nu_{30}^{2},
 \nu_{4}^{2}\cg_{10}\ccs_{17},\nu_{4}\ca_{7}\lt_{8}\ct_{16}, \s\kj'\hc\cg_{29},\s\nu'\hc\ca_{7}\lt_{8}\ct_{16},\mu_{3,4}\cg_{29}\nn$ which is isomorphic to $(\z/2)^{4}\jia\z/4\jia(\z/2)^{2}.$ Then, $P(4\cg_{9}\cg_{16}^{*})=4\nu_{4}\cx''$. By (3), we infer the result.
    
    \item[\rm(4)] By \cite[(7.16), pp.\,69]{Toda}, we derive $P(\cg_{9})= y\nu_{4}\cg'\pm\s\lt'$ for some odd $y$; by \cite[(7.19), pp.\,71]{Toda} we infer $\cg'\nu_{14}\ty \nu_{7}\cg_{10}\m2\nu_{7}\cg_{10}.$ Thus, $\nu_{4}\cg'\nu_{14}\ccs_{17}=\nu_{4}^{2}\cg_{10}\ccs_{17}.$ By \cite[Table I, pp.\,104]{Oguchi}, we have the relation $\lt'\nu_{13}=0\in\pi_{13+3}^
{3}$. Hence, $P(\cg_{9}\nu_{16}\ccs_{19}
)= (y\nu_{4}\cg'\pm\s\lt')\nu_{14}\ccs_{17}=\nu_{4}\cg'\nu_{14}\ccs_{17}=\nu_{4}^{2}\cg_{10}\ccs_{17}.$

   \item[\rm(5)] It follows from (3) and (4).
   
    \item[\rm(6)] Notice in \cite[(9.14), pp.\,144]{Oda} that $\mathrm{Ker}(P: \pi_{38}^{9}\rightarrow\pi_{36}^{4})=\mathrm{span}\n8\cg_{9}\cg_{16}^{*}\nn\tg\z/2.$ Then the result follows from (5).
    \item[\rm(7)] First, $\nu_{5}\ca_{8}\ch_{9}\ct_{17}=P(\ch_{11}\ct_{19})=0$, (the proof of \cite[(9.15), p.\,144]{Oda}). So, the Toda bracket is well-defined. By (6), we have $H \n \nu_{5}\ca_{8}, \ch_{9}\ct_{17},2\e_{37}\nn_{1}=-P^{-1}( \nu_{4}\ca_{7} \ch_{8}\ct_{16})\hc2\e_{38}=8\cg_{9}\cg_{16}^{*}$,\; (a set of a single element).  By (2), \cite[Proposition 3.1\,,\,pp.\,116]{Oda} and $\nu_{9}\cg_{12}\ty 2\cg_{9}\nu_{16}\m 2\nu_{9}\cg_{12}$ (\cite[(7.19), pp.\,71]{Toda}), we have $\nu_{9}^{3}\ct_{18}=8\cg_{9}\cg_{16}^{*}$.
    Notice in \cite[Proposition 4.4\,(6),\,pp.\,122]{Oda} that $H(\nu_{5}\cg_{8}\nu_{15}\ccs_{18})=\nu_{9}^{3}\ct_{18}=8\cg_{9}\cg_{16}^{*}$. Then, the result follows from the exactness of the \textit{EHP} sequence.
      \item[\rm(8)] $\s\pi_{37}^{4}=\mathrm{span}\n \nu_{5}\phi_{8}\cg_{31}, \mu_{4,5}, \ca_{5}\mu_{3,6}\cg_{31}\nn$, (\cite[(9.23), pp.\,145]{Oda}). By the proof of \cite[(9.25), pp.\,145]{Oda}, we derive $\nu_{6}\phi_{9}\cg_{32}=0$. Hence the result holds.
  \end{itemize}  
\end{proof}

\subsection{The element $\cn_{6}'\in\pi_{39}^{6} $ and its order}

Leveraging the groundwork laid out above, we can devise a more refined construction for the element $\cn_{6}'\in\pi_{39}^{6}$, allowing for the precise determination of its order. And we are therefore able to solve Oda's extension problems for the homotopy groups $\pi_{39}^{6}, \pi_{40}^{7} $ and $\pi_{41}^{8}$.

Recall from \cite[Formula (9.16), pp.\,144]{Oda} that the element $\ch_{11}\ct_{19}\in\pi_{39}^{11}$
is of order 2.
\begin{lem} \label{dytd}

Included in $\pi_{39}^{6}$,  the  Toda bracket $$T=\n\s\e_{5}, \s(\nu_{5}\ca_{8}), \s(\ch_{9}\ct_{17}) \nn_{1}$$ is well-defined. And $H(T)=\ch_{11}\ct_{19}\neq0.$

\end{lem}

\begin{proof} By \cite[(5.10),\,pp.\,44]{Toda}, we have $P\e_{11}=[\e_{5},\e_{5}]=\nu_{5}\ca_{8}.$ Recall from the proof of \cite[(9.15), p.\,144]{Oda} that $\nu_{5}\ca_{8}\ch_{9}\ct_{17}=P(\ch_{11}\ct_{19})=0$. So, $T$ is well defined. Then, $H(T)=\ch_{11}\ct_{19}\neq0$ follows from \cite[Proposition 2.6, pp.\,22-23]{Toda}.

\end{proof}

Recall from \cite[(9.25), pp.\,145]{Oda} that
 $\nu_{6}\cg_{9}\nu_{16}\ct_{19}$ is of order 2.

\begin{lem}\label{wdka} Choosing $\cn'_{6}\in T$, we have $H(\cn'_{6})=\ch_{11}\ct_{19}$,
 $2\cn'_{6}=\nu_{6}\cg_{9}\nu_{16}\ct_{19}.$ Successively, $\cn'_{6}$ is of order 4. All properties of the original $\cn'_{6}$ defined in \cite[pp.\,145]{Oda} satisfies, this new choice of $\cn'_{6}$ still  satisfies.
\end{lem}
\begin{proof} By Lemma \ref{dytd}, we see that $H(\cn'_{6})=\ch_{11}\ct_{19}\neq0$ for $\cn'_{6}\in T$. By \cite[Proposition 1.4, pp.\,11]{Toda} and Lemma \ref{qmgx} (7), we have $$2\cn_{6}'\in T\hc 2\e_{39}= \s\e_{5}\hc \s\n \nu_{5}\ca_{8}, \ch_{9}\ct_{17},2\e_{37} \nn_{}\dyd \s\n \nu_{5}\ca_{8}, \ch_{9}\ct_{17},2\e_{37} \nn_{1}\ni\s(\nu_{5}\cg_{8}\nu_{15}\ct_{18}+\s\beta)$$ where $\beta\in \pi_{37}^{4}$.   Notice that $ \mathrm{Ind}\n\nu_{5}\ca_{8}, \ch_{9}\ct_{17},2\e_{37} \nn_{}=[\e_{5},\e_{5}]\hc\pi^{9}_{38}+2\pi_{38}^{5}$ and $2\s\pi_{38}^{5}=0$, (\cite[(9.25),\,pp.\,145]{Oda}).  Then,  $ \mathrm{Ind}(\s\n\nu_{5}\ca_{8}, \ch_{9}\ct_{17},2\e_{37} \nn_{})=2\s\pi_{38}^{5}=0$. So, $2\cn'_{6}=\nu_{6}\cg_{9}\nu_{16}\ct_{19}+\s^{2}\beta.$
Recall from Lemma \ref{qmgx} (8) that
$\s^{2}\beta\in\mathrm{span}\n\mu_{6,6},\;\ca_{6}\mu_{3,7}\cg_{32}\nn.$ So, $$2\cn'_{6}\ty\nu_{6}\cg_{9}\nu_{16}\ct_{19}\m \mu_{6,6},\ca_{6}\mu_{3,7}\cg_{32}.$$
Recall that
$\s^{\wq} \mathrm{span}\n  \mu_{6,6},\ca_{6}\mu_{3,7}\cg_{32}\nn\tg(\z/2)^{2}$ is a direct summand of $\pi_{33}^{S}(S^{0})$, (the proof \cite[(9.27), pp\,146]{Oda}). In fact, $\pi_{33}^{S}(S^{0})\tg(\z/2)^{5},$ (see \cite[pp.\,315]{brwd}). Obviously,  $\s^{\wq}(\nu_{6}\cg_{9}\nu_{16}\ct_{19})=0$, ($\cg\nu=0$, \cite[(7.20), pp.\,72]{Toda}).  Hence, $$2\cn'_{6}=\nu_{6}\cg_{9}\nu_{16}\ct_{19}.$$ We know $\ord(\nu_{6}\cg_{9}\nu_{16}\ct_{19})=2$. Then, $\ord(\cn_{6}')=4.$  Notice the proof of \cite[Formula (9.27), pp.\,145-146]{Oda}. For any element $\mathbbm{x}\in\pi_{39}^{6}$ satisfying $H(\mathbbm{x})=\ch_{11}\ct_{19}$, the original  $\cn'_{6}$ in \cite{Oda} can be taken as $\kappa_{6}'=\mathbbm{x}.$ Hence the result holds.

    \end{proof}
\;\\\indent Following \cite{Oda}, $\cn'_{6+i}:=\s^{i}\cn'_{6}$.
Recall the following facts on the determination of $\pi_{39}^{6}$ from \cite{Oda}. The \textit{EHP} sequence 
\[\xymatrix@C=1.5em{
  \pi_{40}^{11} \ar[r]^{P} &  \pi_{38}^{5} \ar[r]^{\s} & \pi_{39}^{6}\ar[r]^{H} & \pi_{39}^{11}\ar[r]^{P}  & \pi_{37}^{5}
}\]
induces a short exact 
sequence \vspace{-\baselineskip}\[\xymatrix@C=1.5em{
  0 \ar[r]^{} & \s \pi_{38}^{5} \ar[r]^{} & \pi_{39}^{6}\ar[r]^{H\quad} & \mathrm{Im}(H)\ar[r]^{}  & 0,
}\]
where  \vspace{-\baselineskip} \\ \centerline{$\s \pi_{38}^{5}=\mathrm{span}\n \nu_{6}\cg_{9}\nu_{16}\ct_{19},\; \mu_{4,6},\; \ca_{6}\mu_{3,7}\cg_{32}\nn\tg(\z/2)^{3}$}\\\\and \centerline{
$\mathrm{Im}(H)=\mathrm{span}\n C_{1}\w_{23},\;\s F_{1}^{(1)},\;\cg_{11}^{4},\;C_{1}^{(2)},\;\ch_{11}\ct_{19}\nn\tg(\z/2)^{5}.$}\\\\
Moreover, \qquad \qquad \qquad
$H(P(\s A_{1}\hc\w_{25}))= C_{1}\w_{23},\;
H(P(\s A_{1}\hc\cg_{25}\mu_{32}))=\s F_{1}^{(1)},$\;\\\\ \centerline{
$H(\ccs_{6}\cg_{25}^{2})=\cg_{11}^{4}\;\text{and}\; H(P(\s A_{1}^{(2)}))\ty C_{1}^{(2)}\m C_{1}\w_{23},\;\s F_{1}^{(1)},\;\cg_{11}^{4},\;\ch_{11}\ct_{19}.$}\\\\
\noindent The four liftings  
$
P(\s A_{1}\hc\w_{25}), P(\s A_{1}\hc\cg_{25}\mu_{32}) $, $\ccs_{6}\cg_{25}^{2}$ and $P(\s A_{1}^{(2)})$ are all of order 2.

By these facts,  Lemma \ref{wdka} together with \cite[Proposition 9.20, pp.\,144-145]{Oda}, we infer the following proposition, which solves the group extension problems  left by Oda (\cite{Oda}).  
\begin{pro}\label{zdl2} 
    
\begin{eqnarray}
\notag 
\pi_{39}^{6}&=& \mathrm{span}\n \cn'_{6},\; P(\s A_{1}\hc\w_{25}),\;P(\s A_{1}\hc\cg_{25}\mu_{32}),\; \ccs_{6}\cg_{25}^{2},\; P(\s A_{1}^{(2)}),\; \mu_{4,6},\; \ca_{6}\mu_{3,7}\cg_{32}\nn\\\notag
&\tg& \z/4\jia (\z/2)^{6},\quad 2\cn'_{6}=\nu_{6}\cg_{9}\nu_{16}\ct_{19};\\
\notag 
\pi_{40}^{7}&=& \mathrm{span}\n \cg'\ca_{14}\mu_{3,15},\; \cn'_{7},\;  \ccs_{7}\cg_{26}^{2},\; \mu_{4,7},\; \ca_{7}\mu_{3,8}\cg_{33}\nn \\\notag
&\tg& \z/4\jia (\z/2)^{4};\\
\notag 
\pi_{41}^{8}&=& \mathrm{span}\n  \cg_{8}\hc \s^{3}\tau^{\mathrm{IV}},\;\cn'_{8},\;\cg_{8}\ca_{15}\mu_{3,16},\;\cg_{8}\nu_{15}^{2}\ct_{21},\;
\s\cg'\hc\ca_{15}\mu_{3,16},\;  \ccs_{8}\cg_{27}^{2},\\\notag
&& \qquad \; \mu_{4,8},\; \ca_{8}\mu_{3,9}\cg_{34}\nn \\\notag
&\tg&  \z/8\jia\z/4\jia(\z/2)^{6}.\;\end{eqnarray} \;\\\;\vspace{-3.5\baselineskip}\\\indent\qquad\qquad\qquad$\bzd$
\end{pro}
\;\\\indent
 (Without taking the localization), given an integer $k\geq8$, if the stable homotopy group $\pi_{k}^{S}(S^{0})$ is known, then the 8 groups $\pi_{n+k}(S^{n})$ $(n=k+3-i, 1\leq i\leq 8)$ are all known. Notice that the aforementioned 8 groups are the first 8 groups counted backward from the group  which  is first reached in the stable range. This useful result is shown in  \cite{dywx}  presented separately according to the remainders of $k\m8$. Here, we consider the case $k\ty1\m8$ and point out the following.
 
 \begin{lem}\label{c8d}(\cite[Satz\,1.2]{dywx}) Without taking the localization,
   suppose $k\geq1$ and the stable homotopy group $\pi_{8k+1}^{S}(S^{0})=G.$  Then, up to isomorphism, the 8 groups $\pi_{n+8k+1}(S^{n})$ $(n=8k+4-i, 1\leq i\leq 8)$ are given by  $$\pi_{16k+4}(S^{8k+3})=\pi_{16k-1}(S^{8k-2})=\pi_{16k-2}(S^{8k-3})=\pi_{16k-3}(S^{8k-4})= G,$$
$$\pi_{16k+2}(S^{8k+1})=\pi_{16k}(S^{8k-1})= G\jia\z/2,\;\;\pi_{16k+1}(S^{8k})= G\jia(\z/2)^{2}$$ and
$\pi_{16k+3}(S^{8k+2})= G\jia \z.$$\bzd$

 \end{lem}
 
 \indent It is well-known that $\pi_{33}^{S}(S^{0})\tg(\z/2)^{5}$, (\cite[pp.\,315]{brwd}). Recall that $\pi_{33+n}(S^{n})$ ($2\leq n\leq 5$) are shown in \cite[p.\,144-145]{Oda} and  $\pi_{33+9}(S^{9})$ is determined in \cite{32STEM}.  By amalgamating the aforementioned discourse with Proposition \ref{zdl2} and Lemma \ref{c8d}, we arrive at the following corollary. (For $10\leq n\leq 27$,  Y. Hirato, T. Miyauchi, J. Mukai and J. Yang are in the midst of this project).
\begin{table}[h]
\centering\resizebox{14.7cm}{2cm}{
\begin{tabularx}{\textwidth}{|X|X|X|X|X|X|X|X|}
\hline
\textcolor{white}{n,114}$n$ & \textcolor{white}{2,1.11}\texttt{2}   &  \textcolor{white}{3,1110}\texttt{3}   &  \textcolor{white}{4,1441}\texttt{4}   & \textcolor{white}{5,1441}\texttt{5}   & \textcolor{white}{6,1141}\texttt{6}   \\
\hline
\textcolor{white}{0}$\pi_{33+n}(S^{n})$ & $\z/4\jia(\z/2)^{2}$ & \textcolor{white}{3,9.}$(\z/2)^{3}$ & {\small$(\z/8)^{2}\jia (\z/2)^{6}$}&\textcolor{white}{5,99}$(\z/2)^{4}$& \textcolor{white}{}$\z/4\jia (\z/2)^{6}$  \\
\hline
\textcolor{white}{n,114}$n$ & \textcolor{white}{2,11.1}\texttt{7}   &  \textcolor{white}{3,1110}\texttt{8}   &  \textcolor{white}{4,1411}\texttt{9}   & \textcolor{white}{5,.141}\texttt{}   & \textcolor{white}{6,1.41}\texttt{}   \\
\hline
\textcolor{white}{0}$\pi_{33+n}(S^{n})$& \textcolor{white}{}$\z/4\jia (\z/2)^{4}$ & {\small$\z/8\jia\z/4\jia(\z/2)^{6}$} & \textcolor{white}{9,9.}$(\z/2)^{6}$ &  & \\
\hline
\textcolor{white}{n,114}$n$ & \textcolor{white}{2,111}\texttt{28}   &  \textcolor{white}{3,110}\texttt{29}   &  \textcolor{white}{4,411}\texttt{30}   & \textcolor{white}{5,141}\texttt{31}   & \textcolor{white}{6,141}\texttt{32}   \\
\hline
\textcolor{white}{0}$\pi_{33+n}(S^{n})$ &  \textcolor{white}{28,.}$(\z/2)^{5}$ &\textcolor{white}{29,} $(\z/2)^{5}$ &\textcolor{white}{30,} $(\z/2)^{5}$ &\textcolor{white}{31,} $(\z/2)^{6}$ &\textcolor{white}{32,} $(\z/2)^{7}$\\
\hline

\textcolor{white}{n,114}$n$ & \textcolor{white}{2,111}\texttt{33}   &  \textcolor{white}{3,110}\texttt{34}   &  \textcolor{white}{4,411}\texttt{35}   &  \textcolor{white}{4,411}$\cdots$&\textcolor{white}{4,411}$\cdots$ \\
\hline
\textcolor{white}{0}$\pi_{33+n}(S^{n})$ & \textcolor{white}{33,}$(\z/2)^{6}$ & $(\z/2)^{5}\jia\z_{(2)}$& \textcolor{white}{35,} $(\z/2)^{5}$ & \textcolor{white}{4,411}$\cdots$ & \textcolor{white}{4,411}$\cdots$ \\
\hline
\end{tabularx}}
\caption{$\pi_{33+n}(S^{n})$\, ($2\leq n\leq 9$ or $n\geq27$) localized at 2}
\label{t111}

\end{table}
\begin{cor} Localized at 2, the 33-stem homotopy groups $\pi_{33+n}(S^{n})$ ($2\leq n\leq 9$, $n\geq27$) are given by  Table \ref{t111}.\qquad$\bzd$ 
\end{cor} 
\subsection{A problem}

Suppose $p$ is a prime  and  $X$ (without taking the 2-localization) is simply-connected \textit{CW} complex of finite type. After localization at $p$, let $r_{n}(X)=\mathrm{dim}(\pi_{n}(X)\otimes\z/p)$, i.e., the dimension of $\pi_{n}(X)\otimes\z/p$ as a $\z/p$-vector space. It is clear that $r_{n}(X)$ is  an invariant dependent solely on $n$ and the homotopy type of $X$. And it  detects the count of non-trivial direct summands of $\pi_{n}(X)$.

This invariant will afford formidable aid in addressing  extension problems for  homotopy groups, also assists in resolving the group extension problems left by  homotopy spectral sequences  such as Adams spectral sequences.
Notice that Oda's first extension problem can state as in the case  $p=2$, $$r_{39}(S^{6})=7\;\text{or}\;8.$$ 
\textbf{Problem}  How can we analyze the invariant $r_{n}(X)$  from some perspectives, whether geometric or algebraic? Alternatively, how might we construct a new theoretical framework to examine this invariant, thereby assisting in our exploration of  extension problems for homotopy groups?

\end{document}